\documentclass[12pt,letterpaper]{amsart}
\usepackage{amssymb,mathrsfs,url,setspace,xcolor,tikz}
\usepackage[total={7in, 9in},centering]{geometry}
\usepackage[colorlinks,allcolors=blue]{hyperref} 
\usepackage{xpatch} 
\xpatchcmd{\proof} 
{\itshape} 
{\bfseries} 
{}
{}
\newtheorem{theorem}{Theorem}
\newtheorem{proposition}{Proposition}
\newtheorem{lemma}{Lemma}
\newtheorem{corollary}{Corollary}
\newtheorem{definition}{Definition}
\theoremstyle{remark}
\newtheorem{remark}{Remark}
\newtheorem{example}{Example}
\newcommand{\C}{\mathbb{C}}
\newcommand{\supp}{\text{supp}}
\newcommand{\D}{\Omega}
\newcommand{\Dc}{\overline{\Omega}}
\newcommand{\dbar}{\overline{\partial}}
\newcommand{\zb}{\overline{z}}
\newcommand{\wb}{\overline{w}}
\newcommand{\ep}{\varepsilon}
\title[Berezin Regularity of Domains in $\mathbb{C}^n$]{Berezin 
	regularity of domains in $\mathbb{C}^n$ and the essential 
	norms of Toeplitz operators}
\author{\v{Z}eljko \v{C}u\v{c}kovi\'c}
\author{S\"{o}nmez \c{S}ahuto\u{g}lu}
\email{Zeljko.Cuckovic@utoledo.edu, Sonmez.Sahutoglu@utoledo.edu}
\address{University of Toledo, Department of Mathematics \& Statistics, 
Toledo, OH 43606, USA}
\date{\today}

\begin{document}

\begin{abstract}
For the open unit disc $\mathbb{D}$ in the complex plane, it is well 
known that if $\phi \in C(\overline{\mathbb{D}})$ then its Berezin transform 
$\widetilde{\phi}$ also belongs to $C(\overline{\mathbb{D}})$.  We say that 
$\mathbb{D}$ is BC-regular.  In this paper we study BC-regularity of some 
pseudoconvex domains in $\mathbb{C}^n$ and show that the boundary 
geometry plays an important role. We also establish a relationship between 
the essential norm of an operator in a natural Toeplitz subalgebra and its 
Berezin transform. 
\end{abstract}
\maketitle

The Berezin transform plays an important role in operator theory, especially in 
regard to compactness of certain classes of operators.  A prime example is 
a well-known Axler-Zheng theorem \cite{AxlerZheng98} that characterizes 
compactness of Toeplitz operators on the Bergman space of the unit disc.  
The Axler-Zheng theorem, together with its extension done by 
Suarez \cite{Suarez07},  says that an operator $T$ in the Toeplitz algebra 
is compact if and only if the Berezin transform of $T$, denoted by 
$\widetilde{T}(z)$ (to be defined in the next section) goes to 0 as $z$ approaches 
the boundary of the unit disc.  There were many extensions and versions 
of this theorem including the recent contribution 
\cite{CuckovicSahutogluZeytuncu18} for the case the domain 
is pseudoconvex in $\C^n$.

Another example of the importance of the Berezin transform is the result by 
B\'{e}koll\'{e}-Berger-Coburn-Zhu \cite{BekolleBergerCoburnZhu90} that 
characterizes compactness of Hankel operators $H_f$ and $H_{\overline f}$ 
in terms of the mean oscillation $(\widetilde{|f|^2} - |\widetilde{f}|^2)^{1/2}$ 
near the boundary of the unit ball in $\C^n$, where $\widetilde{f}$ and 
$\widetilde{|f|^2}$ are the  Berezin transforms of $T_{f}$ and $T_{|f|^2}$, 
respectively. 

In this paper we are interested in Toeplitz operators acting on pseudoconvex 
domains that are not compact, but we want to find a relationship between 
their essential norms and the sup norms of their Berezin transforms on the 
boundary of the domain.  This led us to the question of the regularity of 
the Berezin transform.  It is well known that on the unit disc $\mathbb{D}$, 
if the function $\phi \in C(\overline{\mathbb{D}})$, then 
$\widetilde{\phi} \in C(\overline{\mathbb{D}})$ and moreover 
$\widetilde{\phi} = \phi$ on $b\mathbb{D}$.  It was shown in \cite{ArazyEnglis01} 
that the second part of this result is true on more general domains. In particular, 
they showed that when $\D$ is a bounded pseudoconvex domain, then 
$\lim_{z\to p} \widetilde{\phi}(z) = \phi(p)$, when $p$ is strongly pseudoconvex 
point in $b\D$ (see also \cite[Lemma 15]{CuckovicSahutogluZeytuncu18}).  
However the important question now is this: Is 
$\widetilde{\phi} \in C(\Dc)$ whenever $\phi\in C(\Dc)$? 
The following two related results are worth mentioning. The first is the 
work by Engli\v{s} \cite{Englis07} where he proved that the Berezin transform 
can have a singularity inside the domain in case the domain is unbounded. 
The second result is due to Coburn \cite{Coburn05} regarding the Lipschitz 
regularity of the Berezin transform of any bounded linear operator 
on the Bergman space of a bounded domain in $\C^n$ the Bergman metric.

In this paper,  we answer the question of regularity of the Berezin transform 
on several classes of bounded domains in $\C^n$ and on a subalgebra of the 
Toeplitz algebra. We show that the answer is negative for bounded convex 
domains  with discs in the boundary as well as dense strongly pseudoconvex 
points (see Theorem \ref{ThmDiscont}). However, the answer is positive for
products of strongly pseudoconvex domains and for bounded convex 
domains with no discs in the boundary (see Theorem \ref{ThmBregular} 
and  \ref{ThmProductDomains}). 
Finally, we establish the relationship between the essential norms 
of operators in certain Toeplitz subalgebra and the sup-norm of the 
Berezin transform on the set of strongly pseudoconvex points.  
The motivation for these results comes from the Axler-Zheng 
Theorem \cite{AxlerZheng98}.  On the unit disk, the compactness of a 
Toeplitz operator is equivalent to the vanishing of its Berezin transform 
on the unit circle.  This suggests that the measure of non-compactness of 
$T$ should be related to the $L^{\infty}$-norm of its Berezin transform 
on the boundary (see Proposition 1 in 
\cite{CuckovicSahutoglu13,CuckovicSahutoglu14IEOT} for a result in this direction).  
For more general domains, we know that the Berezin transform $\widetilde{T}$ 
need not be continuous on the boundary (see Theorem \ref{ThmDiscont} below), 
but we still have $ \widetilde{\phi}(z) \to \phi(p)$, when 
$z$ converges to a strongly pseudoconvex (or finite type) point $p$. 
This suggests that the essential norm of $T$ in the Toeplitz subalgebra should 
be related  to the norm of $\widetilde{T}$, not on the full boundary, 
but on the set of strongly pseudoconvex points. This is precisely what 
we prove in Theorems \ref{ThmEssentialNorm1} and \ref{ThmEssentialNorm2}.

This paper is organized as follows. In the next section we provide the 
basic set up and the main results. In Section \ref{SecProofs} we prove 
the Theorems. In Section \ref{SecExample} we construct a smooth 
bounded convex domain $\D$ in $\C^2$ such that 
$\|\widetilde{T}_{\phi}\|_{L^{\infty}(\D)}<\|T_{\phi}\|_e$ 
for some $\phi\in C^{\infty}(\Dc)$.  

\section{Preliminaries and Results} 
Let $\D$  be a bounded domain in $\C^n$. The space of square 
integrable holomorphic functions on $\D$, denoted by  $A^2(\D)$, 
is called the Bergman space of $\D$. Since $A^2(\D)$ is a 
closed subspace of $L^2(\D)$ there exists a bounded orthogonal 
projection $P:L^2(\D)\to A^2(\D)$, called the Bergman 
projection.  We denote the set of bounded linear operators on 
$A^2(\D)$ by  $\mathscr{B}(A^2(\D))$. The Toeplitz operator 
$T_{\phi}\in \mathscr{B}(A^2(\D))$ with symbol $\phi\in L^{\infty}(\D)$ 
is defined as 
\[T_{\phi}f=P(\phi f)\] 
 for $f\in A^2(\D)$.  We will work  on the norm closed subalgebra 
of $\mathscr{B}(A^2(\D))$ generated by $\{T_{\phi}:\phi\in C(\Dc)\}$ 
and we will denote it as $\mathscr{T}(\Dc)$. 

Next we define the Berezin transform. Let $K(\xi,z)$ denote 
the Bergman kernel of $\D$. We define the normalized kernel as 
\[k_z(\xi)=  \frac{K(\xi,z)}{\sqrt{K(z,z)}}\]  
for $\xi,z\in \D$. The Berezin transform of a bounded linear operator 
$T$ is defined as 
\[\widetilde{T}(z)=\langle Tk_z,k_z\rangle\]
for any $z\in \D$. For $\phi\in L^{\infty}(\D)$ we define 
$\widetilde{\phi}=\widetilde{T}_{\phi}$.  

Let $\mathbb{D}=\{z\in \C:|z|<1\}$ be the unit disc. It is well known that if 
$\phi \in C(\overline{\mathbb{D}})$ then its Berezin transform 
$\widetilde{\phi}$ also belongs to $C(\overline{\mathbb{D}})$.
In this paper we want to study conditions that guarantee continuity of 
the Berezin transform of operators on domains in $\C^n$. To be more specific, 
we want to find necessary and/or sufficient conditions on the domain 
$\D\subset \C^n$ such that $\widetilde{T}\in C(\Dc)$ whenever 
$T\in \mathscr{T}(\Dc)$.

\begin{definition}
Let $\D$ be a bounded domain in $\C^n$. We say that $\D$ is 
\textit{BC-regular} if $\widetilde{T}$ has a continuous extension 
onto $\Dc$ (that is, $\widetilde{T}\in C(\Dc)$) whenever 
$T\in \mathscr{T}(\Dc)$.	
\end{definition}

In this context we should mention the work of Arazy and Engli\v{s} about a weaker 
version of BC-regularity.  More specifically, in \cite[Theorem 2.3]{ArazyEnglis01} 
they proved that if $\D$ is either a bounded domain in $\C$ with
$C^1$-smooth boundary or a strongly pseudoconvex domain in $\C^n$ with
$C^3$-smooth boundary, then $\widetilde{T}_{\phi}\in C(\Dc)$ whenever 
$\phi\in C(\Dc)$. 

BC-regularity makes sense on a large class of domains in $\C^n$ called 
pseudoconvex domains. These domains include convex domains and 
are natural homes for holomorphic functions in the sense that for 
every boundary point $p$ of a pseudoconvex domain $\D$ there 
exists a holomorphic function $f_p$ on $\D$ such that $f_p$ has no 
local holomorphic extension beyond $p$. For $C^2$-smooth domains 
pseudoconvexity is defined as the Levi form of the domain being positive 
semi-definite on the boundary. Similarly, strongly pseudoconvex domains 
are domains whose Levi form is positive definite on the boundary. 
There is a large literature about these domains and we refer the reader to 
\cite{ChenShawBook,KrantzBook,RangeBook} for information  
about pseudoconvexity and general questions in several complex variables. 
Due to an example of Kohn and Nirenberg \cite{KohnNirenberg73} we know 
that the class of pseudoconvex domains is much larger than the locally 
convexifiable  domains.  In the following theorem we show that if the domain 
is a product of strongly pseudoconvex domains such as the polydisc, then the 
domain is BC-regular. We note that strongly pseudoconvex domains are not 
necessarily convex but locally convexifiable (see \cite[Lemma 3.2.2]{KrantzBook}) 
and are well understood in terms of the $\dbar$-Neumann problem and the 
Bergman kernel (see \cite{ChenShawBook,StraubeBook}). 

The following are the main theorems in our paper. These results show 
that the geometry of the domain plays an important role in understanding 
the BC-regularity of the domain. 

Theorem \ref{ThmBregular} gives a positive answer to the question above 
in case the domain is finite type in the sense of D'Angelo or convex 
and its boundary does not contain any analytic structure . We note that 
such domains satisfy property (P) of Catlin or equivalently B-regularity 
of Sibony (see, \cite{Catlin84, Sibony87, FuStraube98,StraubeBook}). 

We say a set $X\subset \C$ contains a non-trivial analytic disc if there exists a 
non-constant holomorphic mapping  $f:\mathbb{D}\to X$.

Finite type domains have been studied in relation to the $\dbar$-Neumann 
problems, behavior of the Bergman kernel, and the CR-geometry, etc.  
There are different notions of finite type but for this paper the most 
relevant is the finite type in the sense of D'Angelo. There is a rich 
literature about this topic. We refer the reader to 
\cite{D'Angelo82,D'AngeloBook,D'AngeloKohn99} and the 
references there in for more information about finite type domains. 

\begin{theorem}\label{ThmBregular}
Let $\D$ be  a bounded convex domain in $\C^n$ with no non-trivial 
analytic disc in the boundary or a $C^{\infty}$-smooth bounded 
pseudoconvex domain in $\C^n$ that is finite type in the sense of 
D'Angelo. Then $\D$ is BC-regular. 
\end{theorem}

The second result shows that cross-products of BC-regular domains are 
BC-regular. 

\begin{theorem} \label{ThmProductDomains}
Let $\D$ be a finite product of bounded BC-regular domains.  
Then $\D$ is BC-regular. 
\end{theorem}

Then we have an immediate corollary to the two theorems above. 
\begin{corollary}\label{Cor1}
Let $\D=\D_1\times \cdots \times \D_m$ such that each $\D_j$ is either 
a $C^2$-smooth bounded strongly pseudoconvex  or a $C^{\infty}$-smooth 
bounded pseudoconvex finite type domain in the sense of D'Angelo. 
Then $\D$ is BC-regular. 
\end{corollary} 

The next theorem gives a negative answer to the BC-regularity 
question on bounded convex domains that contain discs in the 
boundary yet have dense strongly pseudoconvex points. 

\begin{theorem}\label{ThmDiscont}
Let $\D$ be a $C^{\infty}$-smooth bounded convex domain in $\C^n$ for 
$n\geq 2$. Assume that the boundary of $\D$ contains a non-trivial analytic 
disc and the set of strongly pseudoconvex points is dense in the boundary 
of $\D$. Then the Berezin transform does not map $C(\Dc)$ to itself (and 
hence $\D$ is not BC-regular).
\end{theorem}

Next we will introduce the $\dbar$-Neumann problem and its relationship  
to Hankel operators. 

Let $\Box=\dbar^*\dbar+\dbar\dbar^*:L^2_{(0,1)}(\D)\to L^2_{(0,1)}(\D)$ 
be the complex Laplacian, where $\dbar^*$ is the Hilbert space adjoint 
of $\dbar:L^2_{(0,1)}(\D)\to L^2_{(0,2)}(\D)$.   When $\D$ is a bounded 
pseudoconvex domain in $\C^n$, H\"{o}rmander \cite{Hormander65} 
showed that $\Box$ has a bounded inverse $N$ called the 
$\dbar$-Neumann operator. The $\dbar$-Neumann operator is an 
important tool in several complex variables that is closely connected 
to the boundary geometry of the domains. Kohn in \cite{Kohn63} 
showed that the Bergman projection is connected to the $\dbar$-Neumann 
operator by the formula $P=I-\dbar^*N\dbar$. We refer the reader 
to the books \cite{ChenShawBook,StraubeBook} for more information 
about the $\dbar$-Neumann operator. 

Kohn's formula implies that the Hankel operator $H_{\psi}$ 
with symbols $\psi\in C^1(\Dc)$ satisfies the formula 
$H_{\psi}f=\dbar^*N(f\dbar\psi)$ for any $f\in A^2(\D)$. Furthermore, 
one can show that compactness of the $\dbar$-Neumann operator implies 
that $H_{\psi}$ is compact for all $\psi\in C(\Dc)$ 
(see, \cite[Proposition 4.1]{StraubeBook}). We have used the formula 
$H_{\psi}f=\dbar^*N(f\dbar\psi)$ to study compactness of Hankel 
operators on some domains in $\C^n$ on which $N$ is not necessarily 
compact (see, for instance,  
\cite{CuckovicSahutoglu09,CuckovicSahutoglu17,CelikSahutoglu12,Sahutoglu12}). 

In this paper we are also interested in the essential norm estimates of 
operators in   $\mathscr{T}(\Dc)$. Recall that if $T$ is a bounded linear 
map on $A^2(\D)$ then its essential norm is defined  by 
\[\|T\|_e=\inf\{ \|T-K\|: K \text{ is compact on } A^2(\D)\}.\]
In this context we mention the work in \cite{AxlerConwayMcDonald82} where 
the authors were the first to study compactness and the essential norms of 
Toeplitz operators, with symbols continuous up to the boundary, acting 
on the Bergman space of a bounded domain in $\C$. 

The following theorem establishes the relation between the essential norm of 
$T$ and the sup-norm of $\widetilde{T}$ on the set of strongly pseudoconvex 
points under the assumption that the $\dbar$-Neumann operator is compact.

\begin{theorem}\label{ThmEssentialNorm1} 
Let $\D$ be a $C^{\infty}$-smooth bounded pseudoconvex domain 
in $\C^n$ and $\Gamma\subset b\D$ denote the set of finite type 
points in the sense of D'Angelo. Assume that the $\dbar$-Neumann 
operator is compact on $L^2_{(0,1)}(\D)$ and $T\in \mathscr{T}(\Dc)$. 
Then $\widetilde{T}$ has a continuous extension onto 
$\D\cup\Gamma$ and $\|T\|_e=\|\widetilde{T}\|_{L^{\infty}(\Gamma)}$. 
\end{theorem}

In the absence of compactness of the $\dbar$-Neumann operator, we use  
the new essential norm, denoted by $\|.\|_{e*}$, that measures the distance 
of the operator to a new class of operators containing Hankel operators 
with symbols continuous on the closure. This essential norm was introduced in 
\cite{CuckovicSahutogluZeytuncu18}. Next we make the relevant definitions. 

\begin{definition}
Let $\D$ be a $C^2$-smooth bounded pseudoconvex domain in $\C^n$. We say 
\begin{itemize}
\item[i.]  a sequence $\{f_j\}\subset A^2(\D)$ converges to $f\in A^2(\D)$  
	\textit{weakly about strongly pseudoconvex points} if  $f_j\to f$ weakly  
	and $\|f_j-f\|_{L^2(U\cap \D)}\to 0$ for some neighborhood $U$ of the 
	set of the weakly  pseudoconvex points in $b\D$, 
\item[ii.] a bounded linear operator $T$ on  $A^2(\D)$  is 
	\textit{compact about strongly pseudoconvex points} if $Tf_j\to Tf$ in 
	$A^2(\D)$ whenever $f_j\to f$ weakly about strongly pseudoconvex points.
\end{itemize}
\end{definition}

Let $K^*$ denote the operators on $A^2(\D)$ that are compact about 
strongly pseudoconvex points and $T:A^2(\D)\to A^2(\D)$. Then we 
denote 
\[\|T\|_{e*}=\inf\{\|T-S\|: S\in K^*\}.\]

\begin{theorem}\label{ThmEssentialNorm2}
Let $\D$ be a $C^2$-smooth bounded pseudoconvex domain in $\C^n$ 
and $\Gamma$ denote the set of strongly pseudoconvex points in $b\D$. 
Assume that $T\in \mathscr{T}(\Dc)$. Then $\widetilde{T}$ has a continuous 
extension onto $\D\cup \Gamma$ and 
$\|T\|_{e*}=\|\widetilde{T}\|_{L^{\infty}(\Gamma)}$.  
\end{theorem} 	

\section{Proofs of Theorems } \label{SecProofs} 

We start this section by explaining  the role of compactness of 
the $\dbar$-Neumann operator in the following lemma,   
the proof of which is contained in the proof of Theorem 1 
in \cite{CuckovicSahutoglu13}. The basic idea is this. If $\phi_1,\phi_2$ are 
continuous on the closure of the domain on which the $\dbar$-Neumann 
operator is compact, then $T_{\phi_1}T_{\phi_2}=T_{\phi_1\phi_2}+K$,  
where $K$ is a compact operator.  

\begin{lemma}\label{LemFinite} 
Let $\D$ be a bounded pseudoconvex domain in $\C^n$ on which the 
$\dbar$-Neumann operator is compact. Assume that $T$ is a finite sum 
of finite products of Toeplitz operators with symbols continuous on $\Dc$. 
Then there exists $\phi\in C(\Dc)$ and a compact operator $K$ on $A^2(\D)$ 
such that  $T=T_{\phi}+K$.   
\end{lemma}

At this point we would like to clarify the use of weak convergence in this paper. 
Let $\D$ be a domain in $\C^n$ and $\mu$ be a measure supported on $\Dc$.  
We say $|k_z|^2\to \mu$ weakly as $z\to p$ if $\int \phi|k_z|^2\to \int \phi d\mu$ 
as $z\to p$ for all $\phi\in C(\Dc)$. However, we say $k_z\to 0$ weakly in $A^2(\D)$ 
as $z\to p$ if $\langle k_z,f\rangle\to 0$ as $z\to p$ for all $f\in A^2(\D)$.

In several proofs below, we will use the fact that $C(K)$ is a separable Banach  
space, and hence the closed unit ball of its  dual is weak-star metrizable 
(see, for instance, \cite[Theorem 5.1 in Ch V]{ConwayFuncAnaBook}). 
This fact allows us to use sequences  with Alaoglu theorem. 

Next we prove a sufficient condition for BC-regularity below.

\begin{lemma}\label{LemWeakConv}
Let $\D$ be a bounded pseudoconvex domain in $\C^n$ on which 
the $\dbar$-Neumann operator is compact. Assume that   
$k_z\to 0$ weakly in $A^2(\D)$ as $z\to b\D$ and 
$|k_z|^2\to \delta_p$ weakly as $z\to p$ for any $p\in b\D$. 
Then $\D$ is BC-regular.   	
\end{lemma} 	 
\begin{proof}
Assume that $T\in \mathscr{T}(\Dc)$. Then there exists $\{T_j\}$, 
a sequence  of finite sum of finite products of Toeplitz operators with 
symbols continuous on $\Dc$, such that $\|T-T_j\|\to 0$. Since the 
$\dbar$-Neumann operator is compact, by Lemma \ref{LemFinite} we 
have $T_j=T_{\phi_j}+K_j$ for some $\phi_j\in C(\Dc)$ and compact 
operator $K_j$ for every $j$. 
The assumption $|k_z|^2\to \delta_p$ weakly as $z\to p$ for any 
$p\in b\D$ implies that $\widetilde{T}_{\phi_j}=\phi_j$ on $b\D$ for 
every $j$.  Furthermore, since $k_z\to 0$ weakly as $z\to b\D$ and $K_j$ 
is compact we have $\widetilde{K}_j=0$ on $b\D$. Then 
$\widetilde{T}_j\in C(\Dc)$ for every $j$, $\widetilde{T}_j\to \widetilde{T}$ 
in $L^{\infty}(\D)$ as $j\to \infty$, and $\{\widetilde{T}_j\}$ form a 
Cauchy sequence in $C(\Dc)$. Hence $\widetilde{T}$ has a continuous 
extension onto $\Dc$. That is, $\D$ is BC-regular. 
\end{proof}	

\begin{remark}\label{Rmk1}
We note that the condition  $k_z\to 0$ weakly in $A^2(\D)$ as $z\to b\D$ 
in Lemma \ref{LemWeakConv} is automatic if the boundary is $C^{\infty}$-smooth 
(see, \cite[Lemma 4.9]{CuckovicSahutoglu18}). However, it is not automatic if we 
do not assume any regularity of the boundary even if $N$ is compact. 
For instance, if $\D=\{(z_1,z_2)\in \C^2:|z_1|^2+|z_2|^2<1, z_1\neq 0\}$ 
then $k_z\not\to 0$ weakly as $z\to 0$ because the Bergman kernel of 
$\D$ is the same as the Bergman kernel of the ball. However, the  
$\dbar$-Neumann operator  on $\D$ is compact  
(see, \cite[section 4]{FuStraube01} or \cite[pg 99, Example]{StraubeBook}). 
\end{remark}

In relation to Lemma \ref{LemWeakConv}, if we assume a smooth boundary 
then we get a necessary and sufficient condition for BC-regularity under 
the assumption that  the $\dbar$-Neumann operator is compact. 

\begin{proposition} \label{PropCharacterization}
Let $\D$ be a $C^{\infty}$-smooth bounded pseudoconvex domain in 
$\C^n$ on which the $\dbar$-Neumann operator is compact. Then $\D$ 
is BC-regular if and only if $|k_z|^2\to \delta_p$ weakly as $z\to p$ 
for any $p\in b\D$. 
\end{proposition}

\begin{proof}
First let us assume that $\D$ is BC-regular. Let $p\in b\D$ and $\phi\in C(\Dc)$, 
not necessarily holomorphic,  such that $\phi(p)=1$ and $0\leq \phi<1$ on 
$\Dc\setminus \{p\}$. We note that compactness of $N$ implies that the set of 
strongly pseudoconvex points $\Gamma$ is dense in $b\D$ 
(see \cite[Corollary 1]{SahutogluStraube06} and \cite[Corollary 4.24]{StraubeBook}) 
and $\lim_{z\to p}\widetilde{T}_{\phi}=\phi(p)=1$ for any strongly pseudoconvex 
point $p\in b\D$. Since $\widetilde{T}_{\phi}$ is continuous on $\Dc$ and 
it is equal to $\phi$ on $\Gamma$ a dense set in $b\D$ we conclude that 
$\widetilde{T}_{\phi}=\phi$ on $b\D$. Furthermore, since $\phi$ 
peaks at $p$ one can show that $|k_z|^2\to \delta_p$ weakly as $z\to p$. 
This can be seen as follows: If  $|k_z|^2$ does not converge to $\delta_p$ 
weakly then, by Alaoglu Theorem, there exist  a probability measure
 $\mu\neq \delta_p$ and a sequence $\{z_j\}\subset \D$ such that $z_j\to p$ 
and $|k_{z_j}|^2\to \mu$ weakly as $j\to \infty$. Then 
$\widetilde{T}_{\phi}(z_j)\to \int\phi d\mu\neq 1$ as $j\to \infty$ 
which is a contradiction with $\lim_{z\to p}\widetilde{T}_{\phi}=\phi(p)=1$.
	
To prove the converse we use Lemma \ref{LemWeakConv} together 
with that fact that $k_z\to 0$ weakly as $z\to b\D$ if $\D$ is a smooth 
bounded pseudoconvex domain in $\C^n$ (see Remark \ref{Rmk1}).
\end{proof}	
Let $\D$ be a domain in $\D$ and $p\in b\D$. We say $p$ is a peak point for 
$A(\Dc)$, the space of holomorphic functions on $\D$ that are 
continuous on $\Dc$, if there exists $f\in A(\Dc)$ such that $f(p)=1$ and 
$|f|<1$ on $\Dc\setminus \{p\}$. 
\begin{lemma} \label{LemPeak}
Let $\D$ be a bounded domain in $\C^n$ and $p\in b\D$ be a peak 
point for $A(\Dc)$. Then $|k_z|^2\to \delta_p$  weakly as $z\to p$. 
\end{lemma}
\begin{proof} 
Let us choose $\phi\in A(\Dc)$ such that $\phi(p)=0$ and $\phi<0$ on 
$\Dc\setminus \{p\}$.  Then we use \cite[Lemma 15]{CuckovicSahutogluZeytuncu18} 
to conclude that  $\widetilde{T}_{\phi}(z)\to \phi(p)=0$ as $z\to p$ for all 
$\phi\in C(\Dc)$. Then we use the argument as in the proof of 
Proposition \ref{PropCharacterization} with Alaoglu Theorem to conclude 
that $|k_z|^2\to \delta_p$ weakly as $z\to p$.
\end{proof} 

We continue with the proof of Theorem \ref{ThmBregular}. 

\begin{proof}[Proof of Theorem \ref{ThmBregular}] 
First we deal with the case that $\D$ is a convex domain such that there 
is no non-trivial analytic discs in the boundary of $\D$. Then 
\cite[Proposition 3.2]{FuStraube98} implies that each boundary point 
is a peak point  for $A(\Dc)$, the space of holomorphic functions 
that are continuous on $\Dc$, and \cite[Theorem 1.1]{FuStraube98} 
implies that the $\dbar$-Neumann operator is compact. Then 
Lemma \ref{LemPeak} implies that $|k_z|^2\to \delta_p$ weakly as $z\to p$ 
for all $p\in b\D$ and, in turn, Proposition \ref{PropCharacterization} 
implies that $\D$ is BC-regular.  

Next we assume that $\D$ is a $C^{\infty}$-smooth bounded 
finite type domain  in the sense of D'Angelo 
(see, \cite{D'Angelo82,D'AngeloBook}). Then  the $\dbar$-Neumann 
operator is compact (see, \cite{Catlin84}), and one can use \cite{Bell86,Boas87} 
and the proof of \cite[Lemma 1]{CuckovicSahutoglu13} to prove that 
$|k_z|^2\to \delta_p$ weakly as $z\to p$ for any $p\in b\D$. Finally,  
Proposition \ref{PropCharacterization} implies that $\D$ is BC-regular.
\end{proof}

The property that  $|k_z|^2\to \delta_p$ weakly as $z\to p$ clearly plays 
an important role in the proof of Theorem \ref{ThmBregular}. So far in 
this section we showed that this property holds when $p$ is a peak point 
for $A(\Dc)$ (see Lemma \ref{LemPeak}) or it is a finite type point 
(see the proof of Theorem \ref{ThmBregular}). However, these are 
not the only sufficient conditions  as shown by (the proof of) 
Corollary \ref{CorSinglePoint} below. First we prove the following lemma. 

\begin{lemma}\label{LemDelta}
Let $\D$ be a  $C^{\infty}$-smooth bounded pseudoconvex domain 
in $\C^n$ on which  the $\dbar$-Neumann operator is compact.  Assume 
that $W$ is the set of weakly pseudoconvex points and $p\in W$. 
Furthermore, assume that there exists $\{p_j\}\subset \D$ such that 
$|k_{p_j}|^2\to \mu$ weakly as $j\to \infty$ for some measure $\mu$. 
Then $\supp(\mu)\subset W$. 
\end{lemma}
\begin{proof}
We note that Bell's version of Kerzman's result \cite{Bell86} implies 
that if $\D$ has condition R (implied by compactness of $N$) and 
$q$ is a strongly pseudoconvex point, then the Bergman kernel 
$K\in C^{\infty}((U\cap \Dc)\times (V\cap\Dc)$ for small enough 
disjoint neighborhoods $U$ of $p$ and $V$ of $q$. Then 
$|k_z|^2\to 0$ on $V\cap \Dc$ as $z\to p$. Then 
$\supp(\mu)\cap V=\emptyset$. Since $q$ was arbitrary strongly 
pseudoconvex point the support of $\mu$ lies on the set of weakly 
pseudoconvex points. 
\end{proof}

We note that the domains in the following corollary are not 
necessarily convex or finite type domains as in 
Theorem \ref{ThmBregular} (see \cite{KohnNirenberg73}). 
They are not even necessarily locally convexifiable. 

\begin{corollary}\label{CorSinglePoint}
Let $\D$ be a $C^{\infty}$-smooth bounded pseudoconvex domain in 
$\C^n$. Assume that $\D$ has a single weakly pseudoconvex boundary 
point. Then $\D$ is BC-regular.  	
\end{corollary}
\begin{proof}
We will use Sibony's B-regularity (see \cite{Sibony87}). One can show that 
any single point and any compact set of strongly pseudoconvex points are 
B-regular. Then  $b\D$ is B-regular as a compact set that is a countable 
union of B-regular sets  is B-regular (see \cite[Proposition 1.9]{Sibony87}). 
One can also use that fact that a single point has two-dimensional Hausdorff 
measure zero to prove that $b\D$ is B-regular (see \cite{Boas88}).  Then the 
$\dbar$-Neumann operator is compact on $\D$. 

Let $p\in b\D$ be the weakly pseudoconvex point and $\{|k_{z_j}|^2\}$ be a 
sequence where $z_j\to p$ as $j\to \infty$. Then Alaoglu theorem and 
Lemma \ref{LemDelta} imply that there exists a subsequence 
$|k_{z_{j_k}}|^2\to \delta_p$ weakly as $j_k\to \infty$. Hence every 
sequence $\{|k_{z_j}|^2\}$ has a subsequence weakly convergent 
to $\delta_p$. Then we conclude that $|k_z|^2\to \delta_p$ as $z\to p$. 
\end{proof}

Next we present the proof of Theorem \ref{ThmProductDomains}. 
\begin{proof}[Proof of Theorem \ref{ThmProductDomains}]
Let $\D=\D_1\times \cdots \times \D_m$ where each $\D_j$ is a 
bounded BC-regular domain in $\C^{n_j}$. To showcase the idea, 
we will show that $\widetilde{T}_{\phi}\in C(\Dc)$ whenever 
$\phi\in C(\Dc)$. First let us assume that  
\[\phi(z)=z^{\alpha}\zb^{\beta}
=z_1^{\alpha_1}\zb_1^{\beta_1}\cdots z_m^{\alpha_m}\zb_m^{\beta_m}\]  
then, using the fact that 
$k^{\D}_z(\xi)=k^{\D_1}_{z_1}(\xi_1)\cdots k^{\D_m}_{z_m}(\xi_m)$ 
for $z_j,\xi_j\in \D_j$, one can show that  
\[T_{\phi}k^{\D}_z=(T_{z_1^{\alpha_1}\zb_1^{\beta_1}}k^{\D_1}_{z_1}) 
\cdots (T_{z_m^{\alpha_m}\zb_m^{\beta_m}}k^{\D_m}_{z_m}).\]
Therefore, 
\[\widetilde{T}_{z^{\alpha}\zb^{\beta}}= 
(\widetilde{T}_{z_1^{\alpha_1}\zb_1^{\beta_1}}) 
\cdots (\widetilde{T}_{z_m^{\alpha_m}\zb_m^{\beta_m}})\]
and $\widetilde{T}_{z^{\alpha}\zb^{\beta}}\in C(\Dc)$ as 
$\D_j$ is BC-regular (and hence 
$\widetilde{T}_{z_j^{\alpha_j}\zb_j^{\beta_j}}\in C(\Dc_j)$)  for each $j$. 
Then $\widetilde{T}_{\phi}\in C(\Dc)$ for any $\phi$  polynomial 
of $z$ and $\zb$. Finally, we use  Stone-Weierstrass theorem to 
conclude that  $\widetilde{T}_{\phi}\in C(\Dc)$ whenever $\phi\in C(\Dc)$. 

Next we will adopt this basic strategy to prove the general case. 
Let  $\phi_j(z,\zb)=f_{j1}(z_1,\zb_1)\cdots f_{jm}(z_m,\zb_m)$ 
where $f_{jk}\in C(\Dc_k)$ for $j=1,\ldots,m$. Furthermore,  let 
$T=T_{\phi_1}\cdots T_{\phi_m}$. Then using the fact that 
$k^{\D}_z(\xi)=k^{\D_1}_{z_1}(\xi_1)\cdots k^{\D_m}_{z_m}(\xi_m)$ 
for $z_j,\xi_j\in \D_j$ for each $j$, 
one can show that 
\[T_{\phi_j}k^{\D}_z=(T_{f_{j1}}k^{\D_1}_{z_1}) 
\cdots (T_{f_{jm}}k^{\D_m}_{z_m}).\]
Then we compute the Berezin transform of $T$.   
\begin{align*}
\widetilde{T}(z) = &\langle T_{\phi_1} 
\cdots T_{\phi_m}k^{\D}_z, k^{\D}_z \rangle_{\D} \\
=& \langle T_{f_{11}}\cdots T_{f_{m1}}k^{\D_1}_{z_1}, 
k^{\D_1}_{z_1} \rangle_{\D_1} \cdots \langle T_{f_{1m}} 
\cdots T_{f_{mm}}k^{\D_m}_{z_m},
		k^{\D_m}_{z_m} \rangle_{\D_m}. 
\end{align*}
We note that 
$\langle T_{f_{1j}}\cdots T_{f_{mj}}k^{\D_j}_{z_j},k^{\D_j}_{z_j} \rangle_{\D_j} $ 
has a continuous onto $\Dc_j$ as $\D_j$ is BC-regular for each $j$ 
(and $T_{f_{1j}}\cdots T_{f_{mj}}k^{\D_j}_{z_j}\in \mathscr{T}(\Dc)$).  
Hence, $\widetilde{T}\in C(\Dc)$. Therefore, $\widetilde{T}\in C(\Dc)$ whenever 
$T$ if a finite sum of finite products of Toeplitz operators with  symbols  
polynomial of $z$ and $\zb$. Then Stone-Weierstrass theorem implies that 
$\widetilde{T}\in C(\Dc)$ if $T$ is a finite sum of finite products of Toeplitz 
operators with symbols continuous on $\Dc$. This can be seen as follows: Let 
$\{T_j\}$ be a sequence  of finite sum of finite products of Toeplitz operators 
with polynomial symbols  such that $\|T-T_j\|\to 0$ as $j\to\infty$.  
Then $\widetilde{T}_j\in C(\Dc)$ for every $j$, $\{\widetilde{T}_j\}$ form a 
Cauchy sequence in $C(\Dc)$, and $\widetilde{T}_j\to \widetilde{T}$ in 
$L^{\infty}(\D)$ as $j\to \infty$. Hence $\widetilde{T}$ has a continuous 
extension onto $\Dc$ because it is the uniform limit of functions 
continuous on $\Dc$. 

Finally, we use the same argument as in the last paragraph to prove 
that $\widetilde{T}\in C(\Dc)$ whenever $T\in \mathscr{T}(\Dc)$. 
In this case $\{T_j\}$ is a sequence of finite sum of finite products of 
Toeplitz operators with symbols continuous on $\Dc$. Therefore, 
$\D$ is BC-regular. 
\end{proof}

\begin{example}
We note that $\phi$ and  $\widetilde{\phi}$ might not match on $b\D$ 
even if $\widetilde{\phi}$ is continuous up to the boundary.  
One can produce an example on the bidisc as follows. Let 
$\phi(z_1,z_2)=|z_1|^2\in C^{\infty}(\overline{\mathbb{D}^2})$. 
Then  $\phi(0,0)=0$ and by Theorem \ref{ThmProductDomains} 
we have $\widetilde{\phi}\in C(\overline{\mathbb{D}^2})$. However,  
\[\widetilde{\phi}(0,e^{i\theta}) 
	=\frac{1}{\pi}\int_{\mathbb{D}}|z_1|^2dV(z_1)=\frac{1}{2}.\]
Hence $\phi\neq \widetilde{\phi}$ on $b\D$.
\end{example}

Next we present the proof of Theorem \ref{ThmDiscont}. 
\begin{proof}[Proof of Theorem \ref{ThmDiscont}] 
Let us denote $z=(z',z'')\in \C^m\times \C^{n-m}$ and 
\[B_m(z_0',r)=\{z'\in \C^m:\|z'-z_0'\|<r\}\] 
for $1\leq m\leq n-1$ and  $z_0'\in\C^m$. 
  
We first use \cite[Lemma 2]{CuckovicSahutoglu09}  (see also 
\cite[Section 2]{FuStraube98}) to conclude that  there exists 
an $m$-dimensional affine  analytic variety  $\Delta$ in $b\D$ 
for some $1\leq m\leq n-1$. Hence, without 
loss of generality, we assume that  
\[ B_m(0,2\ep)\times\{0\}\subset  \Delta =\{z\in \C^n:z''=0\}\cap b\D\]
for some $\ep>0$. Furthermore, we assume that 
\[\D\subset 
	\{(z_1,\ldots,z_{n-1},x_n+iy_n)\in \C^n:x_n>0, y_n\in\mathbb{R}\},\] 
the negative $x_n$-axis is outward normal direction of $b\D$ on 
$\Delta$, and the set of strongly pseudoconvex points are dense 
in $b\D$. 

Let $\phi(z)=\exp(-\|z'\|^2)$. We note that $0\leq \phi \leq 1 , \phi(0)=1$,  
$\phi(z)<1$ for $z'\neq 0$, and $\phi$ is independent of $z''$. Then, 
by the assumption, there exists a sequence of strongly pseudoconvex points 
$\{p_j\}\subset b\D\setminus \Delta$ such that $\lim_{j\to\infty}p_j=0$ 
and $\widetilde{\phi}(p_j)=\phi(p_j)$ for all $j$. This is a result of 
the fact that at any strongly pseudoconvex boundary point, the Berezin 
transform of $\phi$ equals to $\phi$. Hence we have 
\[\lim_{j\to \infty}\widetilde{\phi}(p_j)  
	= \lim_{j\to \infty}\phi(p_j) =\phi(0)=1.\] 	
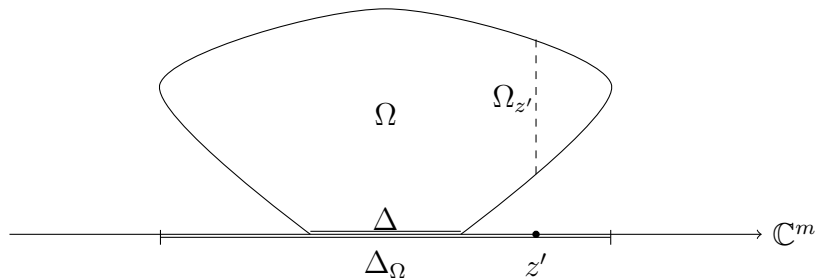
\begin{center}
\begin{tikzpicture}[scale=2]
\draw [->] (-2,0) |- (3,0) node (xaxis) [right] {$\mathbb{C}^m$};
\draw (0.5, 0.8) node {$\Omega$};
\draw (0.5, 0.1) node {$\Delta$};
\draw (0,0.02) -- (1,0.02); 
\draw plot [smooth] coordinates {(0,0) (-1,1) (0.5,1.5) (2,1) (1,0)};	
\draw [|-|](-1,-0.02) -- (2,-0.02); 
\draw (0.5, -0.2) node {$\Delta_{\Omega}$};
\draw (1.35, 0.9) node {$\Omega_{z'}$};
\draw [dashed] (1.5, 0.4) -- (1.5,1.3);
\draw [ultra thick] (1.5,0) circle [radius=0.01];
\draw (1.5, -0.2) node {$z'$};
\end{tikzpicture}	
\end{center}	
Let 
\begin{align*}
 \Delta_{\D} =& \{z'\in \C^m: (z',z'')\in \D \text{ for some } 
z''\in \C^{n-m}\},\\
\D_{z'}= &\{z''\in \C^{n-m}: (z',z'')\in \D\}
\end{align*} 
where $z'\in \Delta_{\D}$, and 
\[\mu_q(z') 
	=\int_{\D_{z'}} |k_q(z',z'')|^2dV(z'')\] 
for $q\in \D$. Then the measure $\mu_qdV$ is a probability 
measure on $\Delta_{\D}$ as 
\[\int_{\Delta_{\D}}\mu_q(z')dV(z') 
	=\int_{\D}|k_q(z)|^2dV(z)=1.\] 

We choose  $q_j=(0,\ldots,0,1/j)$ for $j\in \mathbb{N}$. We extend 
$\mu_qdV$ trivially to a probability measure on $\overline{\Delta_{\D}}$ 
for all $q$. Alaoglu Theorem implies that there exists a probability measure 
$\mu$ on $\overline{\Delta_{\D}}$ such that $\mu_{q_{j_k}}dV\to \mu$ 
weakly as $k\to \infty$ for some subsequence $\{j_k\}$. 

Let $\Delta_0=B_m(0,\varepsilon)$ and $Y_0=\frac{1}{2}\Omega_0\cap B_{n-m}(0,\varepsilon)$. 
Then since $\D$ is convex we have  (see, for instance,  
\cite[Proof of Theorem 2]{ClosCelikSahutoglu18}) 
\[U=\Delta_0\times Y_0 \subset \Omega \cap B_n(0,2\varepsilon)
\subset  4(\Delta_0\times Y_0)=4U.\] 
Let $q\in U$. Then 
\begin{align}\label{EqnFS1} 
\sqrt{K^{\Omega}(q,q)}=k^{\Omega}_q(q) 
=\int_UK^U(q,z)k^{\Omega}_q(z)dV(z) 
\leq \sqrt{K^U(q,q)} \|k^{\Omega}_q\|_{L^2(U)}.
\end{align} 
Then  $\D\cap B_n(0,2\varepsilon)\subset 4U$ and the fact that there exists $C_1>0$ 
such that $K^{\Omega}(q,q)\geq C_1K^{\Omega\cap B_n(0,2\varepsilon)}(q,q)$ 
for $q\in U$ (see \cite[Theorem 6.3.5]{JarnickiPflugBook1stEd}) imply that 
\begin{align}
\|k^{\Omega}_q\|^2_{L^2(U)}\geq &\frac{K^{\Omega}(q,q)}{K^{U}(q,q)} \\
\nonumber =& \frac{K^{4U}(q,q)}{K^{U}(q,q)} 
\frac{K^{\Omega}(q,q)}{K^{\Omega\cap B_n(0,2\varepsilon)}(q,q)} 
\frac{K^{\Omega\cap B_n(0,2\varepsilon) }(q,q)}{K^{4U}(q,q)} \\
\nonumber \geq& C_1 \frac{K^{4U}(q,q)}{K^{U}(q,q)} 
\end{align}
for $q\in U$. 
The transformation formula for the Bergman kernel and \cite[Lemma 4.1]{FuStraube98} 
imply that there exists $C_2>0$ such that $K^{4U}(q_j,q_j)\geq C_2 K^{U}(q_j,q_j)$ 
for sufficiently large $j$. Hence there exists $c_0>0$ such that 
$\|k^{\Omega}_{q_j}\|^2_{L^2(U)}\geq c_0$ for sufficiently large $j$. 

For $z'\in \Delta_0$, let us define 
\[\widehat{\mu}_q(z')=\int_{Y_0}|k^{\D}_q(z',z'')|^2dV(z''). \]
Then $\widehat{\mu}_q$ is a continuous 
(see \cite{BoasStraube91} and \cite[Theorem 6.2.5, Corollary 6.3.11]{ChenShawBook})
and plurisubharmonic function on $\Delta_0$ 
such that $\int_{\Delta_0}\widehat{\mu}_{q_j}(z') dV(z')\geq c_0$ for sufficiently 
large $j$ and  $\widehat{\mu}_q\leq \mu_q$ on $\Delta_0$ for all $q\in U$.
By passing to a subsequence if necessary, we may assume that 
$\widehat{\mu}_{q_{j_k}}dV\to\widehat{\mu}$ weakly on $C(\overline{\Delta}_0)$ 
as $k\to \infty$. Then $\widehat{\mu}(\Delta_0)\geq c_0$. 

 We note that 
\[\widehat{M}_q(r)=\int_{rS^{2m-1}}\widehat{\mu}_q(z')d\sigma(z')\]
is increasing in $r$ where $d\sigma$ is the Euclidean surface area on $rS^{2m-1}$. 
One can conclude this from the standard fact that averages of a plurisubharmonic 
function over concentric spheres are increasing. We include a justification here 
for the convenience of the reader. Let $0<r_1<r_2<\ep$ and $u_{r_2,h}$ be the 
harmonic extension of $\widehat{\mu}_q|_{r_2S^{2m-1}}$ onto the ball 
\[B_m(0,r_2)=\{\xi\in \C^m:\|\xi\|<r_2\}.\] 
Then $\widehat{\mu}_q=u_{r_2,h}$ 
on  $r_2S^{2m-1}$ and $\widehat{\mu}_q\leq u_{r_2,h}$ on $B_m(0,r_2)$. 
Then using the mean value property for harmonic functions in the second 
inequality below we get 
\begin{align*}
\int_{r_1S^{2m-1}}\widehat{\mu}_q(z')d\sigma(z') 
\leq & \int_{r_1S^{2m-1}}u_{r_2,h}(z')d\sigma(z') \\
\leq & \int_{r_2S^{2m-1}}u_{r_2,h}(z')d\sigma(z') \\
= &\int_{r_2S^{2m-1}}\widehat{\mu}_q(z')d\sigma(z').
\end{align*}

To prove that $\mu\neq\delta_0$ it is enough to show that  
$\supp(\widehat{\mu})\neq \{0\}$ because $\widehat{\mu}_q\leq \mu_q$ 
on $\Delta_0$ for all $q$. So next we will show that $\supp(\widehat{\mu})\neq \{0\}$. 
Namely, $\widehat{\mu}$ is not a positive multiple of $\delta_0$, the Dirac 
delta measure centered at 0. 

Let $g(r)=1-2\ep^{-1}r$. Then using the change of variables 
$r\to \ep-r$ on the third integral of the right hand side of the fourth 
equality below, we get
\begin{align*}
\int_{B_m(0,\ep)}g(\|z'\|)\widehat{\mu}_q(z')dV(z')
=& \int_0^{\ep}g(r)\int_{rS^{2m-1}}\widehat{\mu}_q(z')d\sigma(z') dr \\
=& \int_0^{\ep}g(r)\widehat{M}_q(r)dr \\
= & \int_0^{\ep/2}g(r)\widehat{M}_q(r)dr+\int_{\ep/2}^{\ep}g(r)\widehat{M}_q(r)dr\\
= &\int_0^{\ep/2}(g(r)\widehat{M}_q(r)+g(\ep-r)\widehat{M}_q(\ep-r))dr\\
=&\int_0^{\ep/2}g(r)(\widehat{M}_q(r)-\widehat{M}_q(\ep-r))dr\\
\leq &0.
\end{align*}
In the last step above we used the fact that 
$\widehat{M}_q(\ep-r)\geq \widehat{M}_q(r)$ for 
$0\leq r\leq \ep/2$.  Hence, if we substitute $q_{j_k}$ for $q$ in the integrals 
above and letting  $k\to\infty$  we get 
\[\int_{\Delta_0}g(\|z'\|)d\widehat{\mu}(z') \leq 0 
<1=\int_{\Delta_0}g(\|z'\|)d\delta_0\]
Hence, we conclude that $\mu\neq \delta_0$ (in fact, $\mu$ is 
not a multiple of $\delta_0$ either).   

Since $\mu\neq \delta_0$ and $\mu$ is a probability measure on $\overline{\Delta_{\D}}$
we have $\int_{\overline{\Delta_{\D}}} \phi(z',0)d\mu(z')<1$. That is,  
\begin{align*}
\lim_{k\to\infty}\widetilde{\phi}(q_{j_k}) 
= & \lim_{k\to\infty} \int_{\D}\phi(z',z'')|k_{q_{j_k}}(z',z'')|^2dV(z'')dV(z') \\
=& \lim_{k\to\infty} \int_{\Delta_{\D}}\phi(z',0)\mu_{q_{j_k}}(z')dV(z') \\
=& \lim_{k\to\infty} \int_{\overline{\Delta_{\D}}}\phi(z',0)\mu_{q_{j_k}}(z')dV(z') \\
=& \int_{\overline{\Delta_{\D}}}\phi(z',0) d\mu(z') \neq 1.
\end{align*} 
Hence, 
\[\lim_{k\to\infty}\widetilde{\phi}(q_{j_k}) 
	\neq \lim_{j\to\infty}\widetilde{\phi}(p_j).\]
Therefore, $\widetilde{\phi}\not\in C(\Dc)$. 
\end{proof}

\begin{remark}
We note that the measure $\mu$, in the proof of Theorem \ref{ThmDiscont}, 
is the point mass measure when it is centered at a finite type or peak point 
(see Lemma \ref{LemPeak} and proof of Theorem \ref{ThmBregular}) 
and not equal to point mass measure when it is centered  in a disc contained 
in the boundary of a domain satisfying conditions of Theorem \ref{ThmDiscont}.  
It would be interesting to know a characterization of $\mu$ in terms of the 
boundary geometry of smooth bounded pseudoconvex domains. 
\end{remark}

Next we present the proof of Theorem \ref{ThmEssentialNorm1}. 

\begin{proof}[Proof of Theorem \ref{ThmEssentialNorm1}] 
Let  $\Gamma$ denote the set of finite type points in $b\D$. 
First we will show that $\widetilde{T}$ has a continuous extension onto 
$\Gamma$. That is, $\widetilde{T}\in C(\D\cup\Gamma)$. We just need 
to prove continuity at any point in $\Gamma$ as Berezin transform is real 
analytic on $\D$. Since  the $\dbar$-Neumann operator is compact,  
Hankel operators with symbols continuous on the closure of $\D$ are compact 
(see  \cite[Proposition 4.1]{StraubeBook}). If $T$ is a finite sum of finite 
products of Toeplitz operators with symbols continuous on $\Dc$,  
then Lemma  \ref{LemFinite} implies that $T=T_{\phi}+K$ where $\phi\in C(\Dc)$ 
and $K$ is a compact operator. Hence for $T\in  \mathscr{T}(\Dc)$ there exist 
sequences of functions  $\{\phi_j\}\subset C(\Dc)$ and compact operators 
$\{K_j\}$ such that $\|T-T_{\phi_j}+K_j\|\to 0$ as $j\to\infty$. Then 
\[|\widetilde{T}(z)-\widetilde{T}_{\phi_j}(z)+\widetilde{K}_j(z)| 
\leq \|T-T_{\phi_j}+K_j\|\]
for any $z\in \D$. Then $\{\widetilde{T}_{\phi_j}+\widetilde{K}_j\}$ 
is a Cauchy sequence in $L^{\infty}(\D)$. 

Compactness of $K_j$ implies that $\widetilde{K}_j$ has a continuous 
extension up to the boundary of $\D$ and $\widetilde{K}_j= 0$ on $b\D$ 
because $k_z\to 0$ weakly as $z\to b\D$ 
(see \cite[Lemma 4.9]{CuckovicSahutoglu18}). 
Furthermore, just as in the last part of the proof of 
Theorem \ref{ThmBregular}, one can show that $\widetilde{T}_{\phi_j}$ 
has a continuous extension onto  $\D\cup\Gamma$ and 
$\widetilde{T}_{\phi_j}=\phi_j$ on $\Gamma$. Let $p\in \Gamma$ and 
$r>0$ such that $\overline{B(p,r)}\cap b\D\subset \Gamma$ 
(see \cite[Theorem 4.11]{D'Angelo82}). Hence 
$\{\widetilde{T}_{\phi_j}+\widetilde{K}_j\}$ is a Cauchy sequence in 
$C(\Dc\cap\overline{B(p,r)})$ and it converges to $\widetilde{T}$ uniformly 
on the compact set $\Dc\cap\overline{B(p,r)}$. Therefore,  
$\widetilde{T}\in C(\Dc\cap\overline{B(p,r)})$ and since $p\in \Gamma$ 
is arbitrary we conclude that  $\widetilde{T}$ 
has a continuous extension onto $\D\cup\Gamma$. 

We note that compactness of the $\dbar$-Neumann operator implies 
that $\Gamma$ is dense in the boundary (see, for example, 
\cite[Corollary 1]{SahutogluStraube06} or \cite[Corollary 4.24]{StraubeBook}). 
Now we assume that $T$ is a finite sum of finite products of Toeplitz operators 
with symbols continuous on $\Dc$. Then Lemma \ref{LemFinite} implies  
that $T=T_{\phi}+K$ where $\phi\in C(\Dc)$ and $K$ is a compact operator.  
Hence $\widetilde{K}=0$ on $b\D$ and  
\begin{align} \label{Eqn0}
\|T\|_e=\|T_{\phi}\|_e=\|\phi\|_{L^{\infty}(b\D)} 
=\|\phi\|_{L^{\infty}(\Gamma)}
=\|\widetilde{\phi}\|_{L^{\infty}(\Gamma)}
=\|\widetilde{T}_{\phi}\|_{L^{\infty}(\Gamma)}.
\end{align} 
The first equality above comes from the fact that $K$ is compact. 
The second equality is due to \cite[Corollary 3]{CuckovicSahutoglu13}. 
The third equality comes from the fact that $\Gamma$ is dense in $b\D$ 
and the fourth equality is due to the fact that 
$\phi(p)=\widetilde{\phi}(p)$ for any finite type point. 

Finally, we assume that $T\in \mathscr{T}(\Dc)$. For $\ep>0$ there exists 
$\phi\in C(\Dc)$ and a compact operator $K$ such that $\|T-T_{\phi}+K\|<\ep$. 
Then for any compact operator $S$ we have 
\[\|T\|_e\leq \|T+K-S\|\leq \|T-T_{\phi}+K\|+\|T_{\phi}-S\|
\leq \ep+\|T_{\phi}-S\|.\]
If we take infimum over $S$, by \eqref{Eqn0} we get 
\begin{align} \label{Eqn1}
\|T\|_e\leq \|T_{\phi}\|_e+\ep 
= \|\widetilde{T}_{\phi}\|_{L^{\infty}(\Gamma)}+\ep. 
\end{align} 
For any $z\in \D$ we have  
\[|\widetilde{T}(z)-\widetilde{T}_{\phi}(z)+\widetilde{K}(z)|
=|\langle (T-T_{\phi}+K)k_z,k_z\rangle| 
\leq \|T-T_{\phi}+K\|\leq \ep. \]
Then 
\[\|\widetilde{T}-\widetilde{T}_{\phi}\|_{L^{\infty}(\Gamma)}\leq \ep.\]
Combining, the inequality above with \eqref{Eqn1} we get  
\[\|T\|_e \leq  \|\widetilde{T}_{\phi}\|_{L^{\infty}(\Gamma)}+\ep 
\leq  \|\widetilde{T}\|_{L^{\infty}(\Gamma)} 
+\|\widetilde{T}-\widetilde{T}_{\phi}\|_{L^{\infty}(\Gamma)}+\ep 
\leq \|\widetilde{T}\|_{L^{\infty}(\Gamma)}+2\ep.\]
Since $\ep>0$ is arbitrary we have  
\begin{align} \label{Eqn2}
\|T\|_e\leq  \|\widetilde{T}\|_{L^{\infty}(\Gamma)}.
\end{align}

To prove the converse, for $\ep>0$, we choose a compact operator $K$ so that 
\[\|T-K\|\leq \|T\|_e+\ep.\] 
Then for any $z\in \D$ we have 
\[|\widetilde{T}(z)-\widetilde{K}(z)| =|\langle (T-K)k_z,k_z\rangle| 
\leq \|T-K\| \leq \|T\|_e+\ep.\]
Then  $\|\widetilde{T}\|_{L^{\infty}(\Gamma)}\leq \|T\|_e+\ep$ 
(as $\widetilde{K}=0$ on $b\D$) and since $\ep$ is arbitrary 
we get 
\[\|\widetilde{T}\|_{L^{\infty}(\Gamma)}\leq \|T\|_e.\] 
Combining the last inequality with \eqref{Eqn2} we get
$\|T\|_e= \|\widetilde{T}\|_{L^{\infty}(\Gamma)}.$ 
\end{proof}

We finish this section with the proof of Theorem \ref{ThmEssentialNorm2}. 
\begin{proof}[Proof of Theorem \ref{ThmEssentialNorm2}]
Let $T$ be a finite sum of finite products of Toeplitz operators with symbols 
continuous on $\Dc$. Then $T=T_{\phi}+K$ where $K$ is an operator 
that is  compact about strongly pseudoconvex points 
(see proof of Theorem 4 in \cite{CuckovicSahutogluZeytuncu18}) 
and $\phi\in C(\Dc)$. 

Let $\alpha>\|\widetilde{T}\|_{L^{\infty}(\Gamma)}$. Since  $\phi\in C(\Dc)$ and 
\[\|\widetilde{T}\|_{L^{\infty}(\Gamma)} 
=\|\widetilde{T}_{\phi}\|_{L^{\infty}(\Gamma)} 
=\|\phi\|_{L^{\infty}(\Gamma)}\] 
there exists an open neighborhood $U$ of $\Gamma$ such that $
|\phi|<\alpha$ on $U\cap \Dc$. 
We choose a function $\chi\in C^{\infty}_0(U)$ such that $0\leq \chi\leq 1$ 
and $\chi=1$ on a neighborhood of $\Gamma$. One can check that 
$S=T_{(1-\chi)\phi}$ is an operator  compact about strongly pseudoconvex 
points (see \cite[Lemma 12]{CuckovicSahutogluZeytuncu18}). Then  
\[\|T\|_{e*}\leq \|T_{\phi}-S\|\leq \|\chi\phi\|_{L^{\infty}(U)} 
\leq \|\phi\|_{L^{\infty}(U)}\leq \alpha.\] 
Since $\alpha$ is arbitrary we conclude that 
\begin{align} \label{Eqn3} 
\|T\|_{e*}\leq \|\widetilde{T}\|_{L^{\infty}(\Gamma)}.
\end{align}

To prove the converse, let $p\in b\D$ be a strongly pseudoconvex point. 
Then for every $\ep>0$ there exists an operator $S_{\ep}$ compact about 
strongly pseudoconvex points such that 
$\|T-S_{\ep}\|<\|T\|_{e*}+\ep$. Then 
\[\lim_{z\to p}|\widetilde{T}(z)|
=\lim_{z\to p}|\widetilde{T}(z)-\widetilde{S_{\ep}}(z)|
=\lim_{z\to p}|\langle (T-S_{\ep})k_z,k_z\rangle|\leq \|T-S_{\ep}\| 
<\|T\|_{e*}+\ep.\]
In the first equality above we used the fact that $k_z\to 0$ weakly about 
strongly pseudoconvex points as $z\to p$ (see the proof of 
\cite[Theorem 4]{CuckovicSahutogluZeytuncu18}). Hence   
$\|\widetilde{T}\|_{L^{\infty}(\Gamma)}\leq \|T\|_{e*}$. 
Combining  this inequality with \eqref{Eqn3} we conclude that   
\[\|T\|_{e*}=\|\widetilde{T}\|_{L^{\infty}(\Gamma)}.\]

Next if $T\in  \mathscr{T}(\Dc)$ then there exist sequences 
$\{\phi_j\}\subset C(\Dc)$ and  operators compact about strongly 
pseudoconvex points $\{K_j\}$ such that 
\[\|T-T_{\phi_j}+K_j\|\to 0 \text{ as } j\to\infty\]
(again see the proof of Theorem 4 in \cite{CuckovicSahutogluZeytuncu18}). 
Then as in the proof of Theorem \ref{ThmEssentialNorm1} we conclude that 
the sequence 	$\{\widetilde{T}_{\phi_j}+\widetilde{K}_j\}$ is  Cauchy in 
$C(\Dc\cap \overline{B(p,r)})$ for any $p\in \Gamma$ and $r>0$ such that 
$\overline{B(p,r)}\cap b\D\subset \Gamma$ and it converges to 
$\widetilde{T}$ uniformly on $\D$ as $j\to\infty$. Hence $\widetilde{T}$ 
has a continuous extension onto 
$\D\cup \Gamma$. That is, $\widetilde{T}\in C(\D\cup \Gamma)$. 

Then we complete the proof just as in the second half of the proof of 
Theorem \ref{ThmEssentialNorm1} by replacing compact operators 
with operators compact about strongly pseudoconvex points. 
Therefore,  
\[\|T\|_{e*}=\|\widetilde{T}\|_{L^{\infty}(\Gamma)}\] 
for $T\in  \mathscr{T}(\Dc)$.
\end{proof}

\section{An Example}\label{SecExample}
In general, we have the following inequality 
\[\|\widetilde{T}\|_{L^{\infty}(\D)}\leq\|T\|\]
for any $T\in \mathscr{B}(A^2(\D))$. However, there is no relation between 
$\|\widetilde{T}\|_{L^{\infty}(\D)}$ and $\|T\|_e$. This can be seen as follows. 
In Example \ref{ExampleC2} below  we construct a domain $\D$ such that 
$\|\widetilde{T}_{\phi}\|_{L^{\infty}(\D)}<\|T_{\phi}\|_e$ for some 
$\phi\in C^{\infty}(\Dc)$. However,  for any $\phi\in C^{\infty}_0(\D)$ and 
$\phi\not\equiv 0$ we have $\|T_{\phi}\|_e=0<\|\widetilde{T}_{\phi}\|_{L^{\infty}(\D)}$.

\begin{example}\label{ExampleC2}
In this example, we construct a smooth bounded pseudoconvex 
complete Reinhardt domain $\D$ in $\C^2$ and a symbol 
$\phi\in C^{\infty}(\Dc)$ such that 
\[\|\widetilde{T}_{\phi}\|_{L^{\infty}(\Gamma)} 
\leq \|\widetilde{T}_{\phi}\|_{L^{\infty}(\D)}<\|T_{\phi}\|_e
=\|T_{\phi}\|  <\|\phi\|_{L^{\infty}(b\D)}\]
even though $\widetilde{T}_{\phi}\in C(\D\cup \Gamma)$. Let us choose 
an even function $\chi\in C^{\infty}_0(0,1)$ such that $\chi\geq 0$  and 
\begin{align*}
\int_0^1 \chi(r)rdr <2\int_0^1 \chi(r)r^3dr. 
\end{align*}  
One can show the existence of $\chi$ as follows. 
We first note that $\int_{\alpha_0}^1rdr=(1-\alpha_0^2)/2$ and 
$2\int_{\alpha_0}^1r^3dr=(1-\alpha_0^4)/2$. Then for any 
$0<\alpha_0<1$ we have  
\[\int_{\alpha_0}^1rdr=\frac{1-\alpha_0^2}{2} 
< \frac{1-\alpha_0^4}{2} = 2\int_{\alpha_0}^1r^3dr.\]
Then we choose $0<\alpha_0<\alpha<1$ and   $\chi\in C^{\infty}_0(0,1)$ by 
approximating the  characteristic function of $(\alpha_0,1)$ so that 
$\int_0^1 \chi(r)rdr <2\int_0^1 \chi(r)r^3dr$ and 
$\supp(\chi)\subset (\alpha_0,\alpha)$. 
Let $J=[0,\alpha]$. Then 
\begin{align}\label{Eqn+Term}
\int_J \chi(r)rdr <2\int_J \chi(r)r^3dr. 
\end{align}  
	
Now we define the domain $\D$ as follows: First let 
$H\subset [0,1)\times [0,1)$ be a smooth domain in $\mathbb{R}^2$ 
defined by 
\[H=\{(x,y)\in\mathbb{R}^2: 0\leq y<h(x), 0\leq x<1\}.\] 
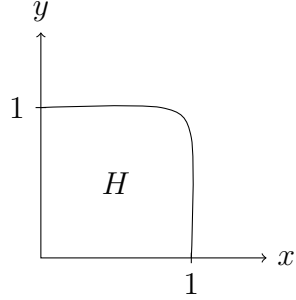
\begin{center} 
\begin{tikzpicture}[scale=2]
\draw [<->] (0,1.5) node (yaxis) [above] {$y$} |- (1.5,0) node (xaxis) [right] {$x$};
\foreach \x in {1} \draw (\x,1pt) -- (\x,-1pt) node[anchor=north] {\x};
\foreach \y in {1} \draw (1pt,\y) -- (-1pt,\y) node[anchor=east] {\y}; 
\draw (0.5, 0.5) node {$H$};
\draw plot [smooth] coordinates { (0,1) (0.8,1) (1,0.8) (1,0)};	
\end{tikzpicture}	
\end{center} 
where $h$ is a smooth function such that $0<h\leq 1$ and 
$J=\{h=1\}\subset [0,1)$. Then 
\[\D=\{(z,w)\in \C^2: (|z|,|w|)\in H\}.\]
We define $\phi(z,w)=\chi(|z|)$ and  
\[\lambda_{nm} 
=\frac{\int_0^1\chi(r)r^{2n+1}(h(r))^{2m+2}dr}{\int_0^1r^{2n+1}(h(r))^{2m+2}dr}
\geq 0\]
for all $n,m=0,1,2,\ldots$. Then one can check that 
\begin{align}\label{Eqn4}
\langle \phi(z,w)z^nw^m-\lambda_{nm}z^nw^m,z^jw^k\rangle =0
\end{align}
for all $n,m,j,k=0,1,2,\ldots$. Then 
\[T_{\phi}z^nw^m=\lambda_{nm}z^nw^m\] 
for all $n\in \mathbb{Z},m=0,1,2,\ldots$. We note that the Bergman kernel 
of $\D$ is 
\[K((\xi,\eta),(z,w)) = 
\sum_{n,m=0}^{\infty} \frac{\xi^n\eta^m}{\delta_{nm}} 
\frac{\zb^n\wb^m}{\delta_{nm}}\] 
where $\delta_{nm}=\|z^nw^m\|$. Then, using \eqref{Eqn4} in the 
second equality below, we get  
\begin{align*}
\langle \phi K(.,(z,w)),K(.,(z,w))\rangle 
= & \sum_{n,m=0}^{\infty} \frac{|z|^{2n}|w|^{2m}}{\delta^4_{nm}} 
\langle \phi(\xi)\xi^n\eta^m,\xi^n\eta^m\rangle \\
=& \sum_{n,m=0}^{\infty} \frac{|z|^{2n}|w|^{2m}}{\delta^4_{nm}} 
\lambda_{nm} \langle \xi^n\eta^m,\xi^n\eta^m\rangle \\
=& \sum_{n,m=0}^{\infty} \frac{|z|^{2n}|w|^{2m}}{\delta^4_{nm}} 
\lambda_{nm}\delta^2_{nm} \\
=&  \sum_{n,m=0}^{\infty} \lambda_{nm} 
\frac{|z|^{2n}|w|^{2m}}{\delta^2_{nm}}. 
\end{align*}
Hence we have 
\[\widetilde{T}_{\phi}(z,w) 
=\frac{\sum_{n,m=0}^{\infty}\lambda_{nm}\frac{|z|^{2n}|w|^{2m}}{\delta^2_{nm}}}
{\sum_{n,m=0}^{\infty}\frac{|z|^{2n}|w|^{2m}}{\delta^2_{nm}}}\]
Next we will compute the norm and essential norm of $T_{\phi}$. 
Since $\chi$ has a compact support in the interior of $J=\{h=1\}$ we have 
\[\lambda_{nm} 
=\frac{\int_J\chi(r)r^{2n+1}dr}{\int_0^1r^{2n+1}(h(r))^{2m+2}dr}\] 
for $n,m=0,1,2,\ldots$ and 
\[\lim_{m\to\infty}\lambda_{nm}=\lambda_{n\infty} 
=\frac{\int_J\chi(r)r^{2n+1}dr}{\int_Jr^{2n+1}dr}<\infty.\] 
Furthermore,  since $0\leq h<1$ on $[0,1]\setminus J$ we have 
\begin{align}\label{EqnSupNorm}
\lambda_{nm} < \lambda_{n(m+1)} < \lambda_{n\infty}
<\|\phi\|_{L^{\infty}(\D)} 
\end{align} 
for all $n,m$.  The last inequality above is due to the fact that 
the probability measure $r^{2n+1}(\int_Jr^{2n+1}dr)^{-1}$ is absolutely 
continuous with respect to the Lebegue measure on $J$ and $\chi$ is a 
non-negative compactly supported function on $J$. Also, 
one can show that the sequence $\{r^{2n+1}(\int_Jr^{2n+1}dr)^{-1}\}$ 
converges to point mass measure $\delta_{\alpha}$ weakly  as $n\to \infty$.   
Hence the fact that the support of $\chi$ is in the interior of $J$ implies that   
\[\lim_{n\to\infty}\lambda_{n\infty}=\lambda_{\infty\infty}=0.\]
Using the fact that the monomials are eigenfunctions 
for $T_{\phi}$ and the set of monomials is a basis for 
$A^2(\D)$ one can show that 
\[\|(T_{\phi}-\lambda I)f\| 
\geq \inf\{|\lambda_{nm}-\lambda|:m,n=0,1,2,\ldots\}\|f\|\]
for all $f\in A^2(\D)$. Hence the spectrum of $T_{\phi}$ is the 
closure of $\{\lambda_{nm}: m,n=0,1,2,\ldots\}$. 
We note that $\lambda_{00}= \widetilde{T}_{\phi}(0,0)$ and 
\eqref{Eqn+Term} implies that 
\begin{align}\label{EqnStrict}
\lambda_{0\infty}<\lambda_{1\infty}.
\end{align}  
Therefore, there exists $n_0\geq 1$ such that  
\[\|T_{\phi}\| = \lambda_{n_0\infty} 
=\max\{\lambda_{n\infty}:n=0,1,2,\ldots \}>0.\] 
Hence 
\begin{align} \label{EqnEigenvalue}
\|\widetilde{T}_{\phi}\|_{L^{\infty}(\D)}
=\sup\left\{\frac{\sum_{n,m=0}^{\infty}\lambda_{nm}
	\frac{|z|^{2n}|w|^{2m}}{\delta^2_{nm}}}
{\sum_{n,m=0}^{\infty}\frac{|z|^{2n}|w|^{2m}}{\delta^2_{nm}}}:(z,w)\in \D \right\} 
<\lambda_{n_0\infty} 
=\|T_{\phi}\|.
\end{align}
The strict  inequality above can be seen as follows: 
Since $\chi$ is supported in $J$ one can show that $\widetilde{T}_{\phi}=0$ 
on $\Gamma\cup \{(z,w)\in b\D:|z|=1\}$.  Therefore, to prove the inequality 
in \eqref{EqnEigenvalue} it is enough to show that 
\[\limsup_{(z_j,w_j)\to(a,b)}\widetilde{T}_{\phi}(z_j,w_j)<\lambda_{n_0\infty}\]
such that  $|b|=1$ and $|a|\in J$.  To that end, let us denote  
\[S(z,w) =\sum_{n,m=0}^{\infty}\lambda_{nm} \frac{|z|^{2n}|w|^{2m}}{\delta^2_{nm}}\] 
and $K(z,w)=\sum_{n,m=0}^{\infty}\frac{|z|^{2n}|w|^{2m}}{\delta^2_{nm}}$. 
We note that $\delta_{(n+N)m}< \delta_{nm}$ for any positive integer $N$. Then 
\[\sum_{m=0}^{\infty}\frac{|w|^{2m}}{\delta^2_{nm}} 
	\leq  \sum_{m=0}^{\infty}\frac{|w|^{2m}}{\delta^2_{(n+N)m}}.\]
First we assume that $a\neq 0$ and we fix a sequence $\{(z_j,w_j)\}$ 
in $\D$ such that $(z_j,w_j)\to(a,b)$ as $j\to\infty$. Then we have  
\begin{align*} 
\frac{|z_j|^{2N}}{K(z_j,w_j)}\sum_{n=0}^{\infty}\sum_{m=0}^{\infty}
\frac{|z_j|^{2n}|w_j|^{2m}}{\delta^2_{nm}}  
\leq \frac{1}{K(z_j,w_j)}\sum_{n=0}^{\infty}
\sum_{m=0}^{\infty}\frac{|z_j|^{2n+2N}|w_j|^{2m}}{\delta^2_{(n+N)m}}.
\end{align*} 
The inequality above together with the fact that 
\[\frac{1}{K(z,w)}\sum_{n,m=0}^{\infty}\frac{|z|^{2n}|w|^{2m}}{\delta^2_{nm}} =1\] 
imply that 
\begin{align}\label{EqnLimInf}
\liminf_{j\to \infty}\frac{1}{K(z_j,w_j)}\sum_{n=N}^{\infty}\sum_{m=0}^{\infty}
	\frac{|z_j|^{2n}|w_j|^{2m}}{\delta^2_{nm}}  \geq |a|^{2N}>0.
\end{align} 
We choose $N$ so that $\lambda_{n\infty}<\lambda_{n_0\infty}/2$ 
for $n\geq N$. Then 
\begin{align*}
\frac{S(z_j,w_j)}{K(z_j,w_j)} 
\leq \frac{\lambda_{n_0\infty}}{K(z_j,w_j)}\sum_{n=0}^{N-1}\sum_{m=0}^{\infty}
\frac{|z_j|^{2n}|w_j|^{2m}}{\delta^2_{nm}}  
+\frac{\lambda_{n_0\infty}}{2K(z_j,w_j)}\sum_{n=N}^{\infty}\sum_{m=0}^{\infty}
\frac{|z_j|^{2n}|w_j|^{2m}}{\delta^2_{nm}}.
\end{align*} 
That is, $S(z_j,w_j)/K(z_j,w_j)$ is bounded from above by a number 
that is a convex combination of $\lambda_{n_0\infty}/2$ and 
$\lambda_{n_0\infty}$ and \eqref{EqnLimInf} implies that in the limit 
the upper bound is strictly less than $\lambda_{n_0\infty}$. Therefore,  we have
\[\limsup_{j\to\infty}\frac{S(z_j,w_j)}{K(z_j,w_j)}<\lambda_{n_0\infty}.\]  
In case $a=0$, we argue as follows: One can show that 
$\delta^2_{0m}\leq \pi^2/(m+1)$ and, by integrating $z^nw^m$ on 
$\{|z|<\alpha\}\times \mathbb{D}$,  we have 
$1/\delta^2_{nm}\leq (n+1)(m+1)/\alpha^{2n+2}\pi^2$. Then for every 
$\ep>0$ there exists $\eta>0$ such that $(n+1)/\alpha^{2n+2}<(\ep/\eta)^{2n}$ 
for all $n$. Then  $\eta^{2n}\delta^2_{0m}/\delta^2_{nm}<\ep^{2n}$ 
for all $n=1,2,3,\ldots$ and all $m$. Hence we have  
\[\frac{|z|^{2n}|w|^{2m}}{\delta^2_{nm}} 
\leq   \frac{\ep^{2n}|w|^{2m}}{\delta^2_{0m}}\] 
for any $|z|<\eta$ and all $n=1,2,3,\ldots$ and all $m$. Therefore,  
for $|z|<\eta$ we have 
\begin{align*}
\frac{S(z,w)}{K(z,w)}
\leq & \frac{1}{K(z,w)}\sum_{n,m=0}^{\infty} 
\frac{\lambda_{nm}\ep^{2n}|w|^{2m}}{\delta^2_{0m}} \\
\leq&   \frac{\sum_{m=0}^{\infty}\lambda_{0m}\frac{|w|^{2m}}{\delta^2_{0m}}
	+\sum_{m=0}^{\infty}\sum_{n=1}^{\infty}\lambda_{nm}\ep^{2n} 
	\frac{|w|^{2m}}{\delta^2_{0m}}}
{\sum_{m=0}^{\infty}\frac{|w|^{2m}}{\delta^2_{0m}}} \\
\leq &\lambda_{0\infty}+\lambda_{n_0\infty}\sum_{n=1}^{\infty}\ep^{2n}.
\end{align*}
Now using  \eqref{EqnStrict} and the fact that $\varepsilon > 0$ above 
is arbitrary, we conclude that 
\[\limsup_{(z,w)\to (0,b)}\frac{S(z,w)}{K(z,w)} 
\leq \lambda_{0\infty}<\lambda_{1\infty}\leq \lambda_{n_0\infty}.\] 
Therefore, and the strict inequality in \eqref{EqnEigenvalue} is verified. 
	
Since $\widetilde{T}_{\phi}$ has continuous extension to $\Gamma$ we have 
\[ \|\widetilde{T}_{\phi}\|_{L^{\infty}(\Gamma)} 
\leq  \|\widetilde{T}_{\phi}\|_{L^{\infty}(\D\cup\Gamma)}
=\|\widetilde{T}_{\phi}\|_{L^{\infty}(\D)}.\] 
We note that each $\lambda_{n\infty}$ is in the essential spectrum 
because they are limits of eigenvalues and $T_{\phi}$ is self-adjoint 
(see \cite[Theorem 1.6 in ch IX]{EdmundsWvansBook}).  
So  $\|T_{\phi}\|_e=\lambda_{n_0\infty}$ and, therefore, using 
\eqref{EqnSupNorm} and \eqref{EqnEigenvalue} we get  
\[ \|\widetilde{T}_{\phi}\|_{L^{\infty}(\Gamma)} 
\leq  \|\widetilde{T}_{\phi}\|_{L^{\infty}(\D)}<\|T_{\phi}\|_e
=\|T_{\phi}\|
=\lambda_{n_0\infty}  
<\|\phi\|_{L^{\infty}(b\D)}=\|\chi\|_{L^{\infty}(J)}.\]
\end{example}

\section{Acknowledgment} 
The authors are indebted to the referee for reading the paper carefully, 
for pointing out a gap in the earlier version of the paper and for suggestions 
that improved Theorem \ref{ThmProductDomains} and the overall presentation. 
They would also like to thank Trieu Le who noticed an inaccuracy in 
Example \ref{ExampleC2}. 


\end{document}